\documentclass[12pt,leqno]{article}
\usepackage[active]{srcltx}
\usepackage{amsmath, amsthm, amsfonts, amssymb, color}
\usepackage[matrix,arrow]{xy}
\usepackage{mathrsfs}
\usepackage{enumerate}
\usepackage{amsopn} 
\usepackage{bbm} 
\newcommand{\df}{\mathcal{E}}
\newcommand{\dom}{\mathcal{D}}
\newcommand{\E}{\mathbb{E}}

\newcounter{cprop}[section]

\newcommand{\ti}[1]{\tilde{ #1 }}
\newcommand{\mcB}{\mathcal{B}}

\newtheorem{hypothesis}[cprop]{Hypothesis}
\newtheorem{definition}[cprop]{Definition}
\newtheorem{remark}[cprop]{Remark}
\newtheorem{lemma}[cprop]{Lemma}
\newtheorem{proposition}[cprop]{Proposition}
\newtheorem{theorem}[cprop]{Theorem}
\newtheorem{corollary}[cprop]{Corollary}
\newtheorem{example}[cprop]{Example}

\usepackage{mdef}

\title{\bf The global random attractor for a class of stochastic porous media equations}

\author{
\textbf{Wolf-J\"urgen Beyn}
\thanks{Supported by the DFG through SFB-701.}
\\
\small{Faculty of Mathematics, University of Bielefeld, Germany} \\
\small{beyn@math.uni-bielefeld.de} 
\\
\textbf{Benjamin Gess}
\thanks{Supported by DFG-Internationales Graduiertenkolleg Stochastics and Real World Models, the SFB-701 and the BiBoS-Research Center.} \\
\small{Faculty of Mathematics, University of Bielefeld, Germany} \\
\small{bgess@math.uni-bielefeld.de}
\\
\textbf{Paul Lescot }
\\
\small{Laboratoire de Math\'ematiques Rapha\"el Salem, CNRS, UMR 6085, Universit\'e de Rouen, France} \\
\small{paul.lescot@univ-rouen.fr}
\\
\textbf{Michael R\"ockner }
\thanks{Supported by the DFG through SFB-701 and IRTG 1132 as well as the BIBOS-Research Center. } \\
\small{Faculty of Mathematics, University of Bielefeld, Germany}\\ 
\small{and}\\
\small{Department of Mathematics and Statistics, Purdue University, U.S.A.} \\
\small{roeckner@math.uni-bielefeld.de}
}

\date{February 7 2010}

\begin{document}
 
\maketitle

\begin{abstract}
 We prove new $L^2$-estimates and regularity results for generalized porous media equations ``shifted by'' a function-valued Wiener path. To include Wiener paths with merely first spatial (weak) derivates we introduce the notion of ``$\zeta$-monotonicity'' for the non-linear function in the equation. As a consequence we prove that stochastic porous media equations have global random attractors. In addition, we show that (in particular for the classical stochastic porous media equation) this attractor consists of a random point.
\end{abstract}

\noindent {\bf 2000 Mathematics Subject Classification AMS}: Primary: 76S05, 60H15, Secondary: 37L55, 35B41.

\noindent {\bf Key words }: Stochastic porous medium equation, Random attractor, Random dynamical systems, Stochastic partial differential equations.

\bigskip
\maketitle


\section{Introduction}
In recent years there has been quite an interest in random attractors for stochastic partial differential equations. We refer e.g. to \cite{cgs2008}, \cite{clr}, \cite{cs2005}, \cite{cf2}, \cite{cf1}, \cite{fs1996}, \cite{KL07}, \cite{s1991}, but this list is far from being complete. The study of a new class of stochastic partial differential equations, namely stochastic porous media equations was initiated in \cite{dpr2004} and further developed in \cite{dprrw}, as well as in a number of subsequent papers (see Sect. 1 below for a more complete list). So far, however, random attractors for stochastic porous media equations have not been investigated.\\ \\
The purpose of this paper is to analyze or even determine the random attractor (in the sense of \cite{c2}, \cite{cf2}, \cite{cf1}) of a stochastic porous medium equation over a bounded open set $\Lambda \subset \R^d$ of type 
\begin{equation}\label{e1.0}
dX_t = \Delta(\Phi(X_t))dt + QdW_t, \ t\ge s, 
\end{equation}
where  $t,s \in \R$, $\Delta$ is the Laplacian, $\Phi: \mathbb{R} \rightarrow \mathbb{R}$ is continuous, $\Phi(0)=0$, and $\Phi$ satisfies certain coercivity conditions and $(W_t)_{t \ge 0}$ is a function valued Wiener process on a probability space $(\O,\F,P)$.
\vskip10pt
To state our results precisely, we need to recall some of the underlying notions and describe the set-up.
This we shall do in Section 1 below. Here we only briefly describe some of the main analytic results we have obtained and which are crucial for the probabilistic part, more precisely, for the proof of the existence of a global (compact) random attractor for \eqref{e1.0}.
\vskip10pt
\noindent As explained in detail in the next section a fundamental property to be established is the cocycle property for the random dynamical system (cf. \cite{a1,IS01}) given by the solutions to \eqref{e1.0} for all $\o \in \O$ (outside a set of $\P$-measure zero), all times $s,t \in \R$ and all initial conditions $x \in H$ (= the Hilbert space carrying the solution-paths to \eqref{e1.0}).\\
Therefore, we have to restrict to additive noise and transform equation \eqref{e1.0} by the usual change of variables
\begin{equation*}
 Z_t := X_t - QW_t (\o)
\end{equation*}
to the equation
\begin{equation}\label{0.3}
 dZ_t = \Delta \Phi (Z_t + QW_t(\o))dt, \quad t \ge s,
\end{equation}
for $\o \in \O$ fixed, i.e. to a deterministic partial differential equation with time dependent nonlinear coefficient and fixed parameter $\o \in \O$. The analysis of this equation is hence purely analytic. Our main results are the regularity Lemma 3.3 and the estimate on the $L^2$-norm of the solution to \eqref{0.3} in Theorem 3.1. These results are crucial for the existence proof of a random attractor for \eqref{e1.0} and in particular the latter gives an explicit control of the $\o$-dependence. To get this estimate on the $L_2$-norm of the solution to \eqref{0.3} we introduce the new notion of ``$\zeta$-weak monotonicity'' (cf. Hypothesis 1.1 below) for the function $\Phi$, which seems to be exactly appropriate for our purposes. We distinguish two cases, namely $QW_t \in H_0^{2,p+1}(\L)$ and the much harder case when $QW_t \in H_0^{1,p+1}(\L)$. For details we refer to Sections 2 and 3 below. We would, however, like to emphasize that these analytic results are of interest in themselves and bear potential for further applications besides merely the analysis of random attractors.\\
On the basis of the estimates obtained in Sections 2 and 3 we can then use a standard result from \cite{cf1} to prove the existence of a global (compact) random attractor for \eqref{e1.0} in Section 4.\\
In Section 5 under a different (more restrictive, see Remark \ref{rmk:strong_monotone}) set of assumptions on $\Phi$ we prove that the random attractor exists and is just a random point by a different, but very direct technique. We conclude this paper by some short remarks on computational methods in Section 6.

\section{Basic notions and framework}
Equation \eqref{e1.0} has recently
been extensively studied within the so-called variational approach to SPDE 
(cf.\ e.g.\ \cite{bdpr}, \cite{bdpr2}, \cite{bdpr3}, \cite{bdpr4}, \cite{dprrw}, \cite{kim2006}, \cite[Example 4.1.11]{pr}, \cite{jrw}, \cite{rw}, \cite{wang2007}, we also refer to \cite{a86}, \cite{v07} and the references therein as background literature for the deterministic case).
The underlying Gelfand triple is
\begin{equation}\label{eq-pp}
  V \subset H \subset V^{*},\tag{1.1}
\end{equation} 
where $V := L^{p+1}(\L)$, $p \ge 1$, $H := H^1_0(\Lambda)^*$, with $H^1_0(\Lambda)$ being the Sobolev space of order one on $\Lambda$ with Dirichlet boundary conditions. We emphasize that the dualization in \eqref{eq-pp} is with respect to $H$, i.e. precisely
  \[ V \subset H \equiv H^* (= H^1_0(\Lambda)) \subset V^*, \]
where the identification of $H$ and $H^*$ is given by the Riesz isomorphism,
$\norm u^2_{H^1_0} := \int_\Lambda |\nabla u|^2_{\R^d}d\xi$, $u \in
  H^1_0(\Lambda)$, and $\norm\cdot _H$ is its dual norm.
Here $|{\cdot}|_{\R^d}$ denotes Euclidian norm on $\R^d$ and below $\bracket{\cdot,\cdot}_{\R^d}$ shall denote the corresponding inner product. By $\norm{\cdot}_p$ we will denote the $L^p$-norm. Further, for $r \in \N, p \ge 1$ let $H_0^{r,p}(\L)$ denote the usual Sobolev space of order $r$ in $L^p(\L)$ with Dirichlet boundary conditions.

\vskip10pt
We take $Q$ and the Wiener process $W_t$ of the following special type. $W= (\beta^{(1)}, \ldots , \beta^{(m)})$ is a Brownian motion on $\R^m$ defined on the canonical Wiener space $(\O, \mcF, (\mcF_t ), P)$, i.e. $\Omega := C(\R_+, \R^m)$, $W_t(\omega):= \omega (t)$, and $(\F_t)$ is the corresponding natural filtration. As usual we can extend $W_t $ (and $\mcF_t$) for all $t\in \R$ (cf. e.g. \cite[p. 99]{pr}). $Q: \R^m \to H$ is defined by 
   \[Qx = \sum_{j=1}^m  x _ j \vp _j , \quad x = (x_1, \ldots ,x_m )\in \R^m,\]
for fixed $\vp_1 , \ldots , \vp _m \in \C_0^1 (\L) \;(\subset L^2 (\L )\subset H)$. Here $\C_0^1(\L)$ denotes the set of all continuously differentiable functions with compact support in $\L$. 

\vskip10pt
The existence and uniqueness of solutions for \eqref{e1.0} under monotonicity and coercivity conditions on $\Phi$ is well-known even under much more general conditions than will be used here (see \cite{jrw}, \cite{bdpr2}). We will always assume the continuous function $\Phi: \mathbb{R} \rightarrow \mathbb{R}$ to satisfy the following conditions:
\begin{enumerate}
 \item[(A1)] Weak monotonicity: For all $t,s \in \mathbb{R}$
   \begin{equation*}\label{phi_conditions_1_1}
      (\Phi(t)-\Phi(s))(t-s) \ge 0.
   \end{equation*}
  \item[(A2)]\label{phi_conditions_1_2} Coercivity: There are $p \in [1,\infty), a \in (0,\infty), c \in [0,\infty)$ such that for all $s \in \mathbb{R}$
    \begin{equation*}
      \Phi(s)s \ge a|s|^{p+1} - c.
    \end{equation*}
  \item[(A3)] Polynomial boundedness: There are $c_1, c_2 \in [0,\infty)$ such that for all $s \in \mathbb{R}$
    \begin{equation*}\label{phi_conditions_1_3}
          |\Phi(s)| \le c_1 |s|^p + c_2,
    \end{equation*}
    where $p$ is as in (A2).
\end{enumerate}

Here and below the notion of solution to \eqref{e1.0} is the usual one (cf. \cite[Definition 4.1]{pr}). We recall that in particular 
\begin{equation}
\E \int_0^T \|X_t\|^{p+1}_{p+1} \; dt < \infty \text{ for all } T > 0,
\end{equation}
with $p$ as in (A2), (A3).

In order to obtain the existence of a random attractor we need slightly more restrictive dissipativity and coercivity conditions on $\Phi$. We 
will prove existence under two sets of assumptions. In the first case we need to assume stronger regularity of the noise, i.e. 
$QW_t \in C_0^2(\L)$, while in the second we allow $QW_t \in C_0^1(\L)$, but require stronger assumptions on the non-linearity $\Phi$.

\begin{hypothesis}\label{phi_conditions_2}
Assume $\vp_j \in C_0^2(\L)$, $1 \le j \le m$, thus $QW_t \in C_0^2(\L)$. Let further $\zeta: \mathbb{R} \rightarrow \mathbb{R}$, $\zeta(0)=0$ be a function such that
\begin{enumerate}
 \item[(A1)'] $\zeta$-Weak monotonicity: For all $t,s \in \mathbb{R}$
       \begin{equation*}\label{phi_conditions_2_2}
         (\Phi(t)-\Phi(s))(t-s) \ge (\zeta(t) - \zeta(s))^2.
       \end{equation*}
 \item[(A2)'] $\zeta$-Coercivity: For $p, a, c$ as in (A2) and for all $s \in \mathbb{R}$
      \begin{equation*}\label{phi_conditions_2_1}
          \zeta(s)^2 \ge a|s|^{p+1} - c.
       \end{equation*}
\end{enumerate}
\end{hypothesis}

\begin{remark}\label{rmk:A1'->A2'}
Note that we do not assume $\zeta$ (hence $\Phi$) to be stricly monotone. Furthermore, we note that the first inequality in (A2)' follows from (A1)' since 
$\Phi(0)=0=\zeta (0)$.
\end{remark}

\begin{remark}\label{rmk:phi_diff}
In case of a continuously differentiable nonlinearity $\Phi$, (more precisely, it suffices to assume that $\Phi \in H^{1,1}_{loc}(\mathbb{R})$) it is easy to find a candidate for $\zeta$. Namely, we simply define 
\begin{equation}\label{0.1}
\zeta(s) := \int_0^s \sqrt{\Phi'(r)}dr, \; s \in \R.
\end{equation}
Then by H\"older's inequality (A1)' holds. Therefore, to ensure that also (A2)' holds we only need to assume
that for some $a \in (0,\infty), c \in [0,\infty), p \in [1,\infty)$
\begin{equation}\label{1.3}
\left( \int_0^s \sqrt{\Phi'(r)} dr \right)^2 \ge a|s|^{p + 1} - c \quad \forall s \in \R.
\end{equation}
Conversely, this produces a lot of examples for $\Phi$ satisfying (A1)',(A2)',(A3). Simply, take $\zeta : \R \to \R$ continuously differentiable and non-decreasing with
$\zeta (0)=0$ and such that for some $a \in (0,\infty); c,c_1,c_2 \in [0,\infty), p \in [1,\infty),$ \[ \zeta^2(s) \ge a |s|^{p+1}-c, \quad
\zeta'(s) \le c_1 |s|^{\frac{p-1}{2}}+c_2 \quad \forall s \in \R.\] Then define \[\Phi(s):= \int^s_0 (\zeta'(r))^2 dr, \; s \in \R.\] In particular,
$\Phi(s):=s|s|^{p-1}$ arises this way (cf. also Section 5 below). In this case we have $\zeta(s) = \left( \frac{2\sqrt{p}}{p+1} \right) s|s|^{\frac{p-1}{2}}$.
\end{remark}

\begin{hypothesis}\label{phi_conditions_3} Let $\vp_j \in C_0^1(\L), 1 \le j \le m$. Assume further that $\Phi \in C^1(\R)$, satisfying \eqref{1.3}
such that
\begin{equation}\label{1.4}
 \Phi'(r) > 0 \text{  for almost all } r \in \mathbb{R},
\end{equation}
and that for some $\ti{c}_1 \in [0,\infty)$ 
\begin{equation}\label{1.5}
 \Phi'(s) \le \ti{c}_1 (|s|^{p-1} + 1) \quad \forall s \in \R,
\end{equation}
    where $p$ is as in \eqref{1.3}.
\end{hypothesis}

\begin{example}\label{exam:Phi}
  Let $\d > 0$ and
  $$ \Phi(r):= 
    \begin{cases}
        (r+\d)^3  & \text{, for } r \le -\d \\
        0         & \text{, for } |r| < \d \\
        (r-\d)^3  & \text{, for } r \ge \d.
    \end{cases} $$
  Then $\Phi \in H^{1,1}_{loc}(\R)$ and by Remark \ref{rmk:phi_diff} Hypothesis \ref{phi_conditions_2} holds with $p = 3$. Hypothesis \ref{phi_conditions_3}, however, does not hold.
\end{example}

In Remark \ref{rmk:phi_diff} we have already obtained
\begin{lemma}\label{lemma:hyp1.4->hyp1.1}
  Assume $\varphi_j \in C^2_0(\L), 1 \le j \le m$. Then Hypothesis \ref{phi_conditions_3} implies Hypothesis \ref{phi_conditions_2} with
    \[ \zeta(s) := \int_0^s \sqrt{\Phi'(r)}dr, \; s \in \R. \]
\end{lemma}


\begin{remark}\label{1.1}
\begin{enumerate}
\item[(i.)]\label{1.1.i} There is a set $\O_0 \subset \O$ of full measure such that for each $p\geq 1$, $\o\in \O_0$, $\norm{ QW_t (\o)}^p_{p}$\;, $\norm{\nabla ( QW _t (\o))}_{p}^p$ and (if $QW_t \in C_0^2$) $\norm{\D ( QW _t (\o))}_{p}^p$ are of at most polynomial growth as ${|t|} \to \infty$.
\item[(ii.)] We shall largely follow the strategy of \cite{cf1}, in which similar assumptions on $Q$, hence on the noise $QW$ are made. The condition that each $\vp_i$ should be in $C_0^1(\L)$ ($C_0^2(\L)$ resp.) can be easily relaxed to $QW_t \in H_0^{1,p+1}(\L)$ ($QW_t \in H_0^{2,p+1}(\L)$ resp.) and is imposed here for the sake of simplicity only.
\end{enumerate}
\end{remark}

In the following let $\l_1$ denote the constant appearing in Poincar\'e's inequality, i.e. for all $f \in H_0^{1,2}(\L)$ 
\[\l_1 \int _\L f (x)^2 dx \leq \int _\L | \nabla f (x)| ^2 dx.\]
For $t\geq s$ and $x\in H$, $X(t,s,x)$ will denote the value at time $t$ of the solution $X_t$ of \eqref{e1.0} such that $X_s = x$. 

\vskip10pt
We now recall the notions of a random dynamical system and a random attractor. For more details confer \cite{a1,cf2,cf1}. Let $((\O,\mcF,\P),(\theta_t)_{t \in \R})$ be a metric dynamical system over a complete probability space $(\O,\mcF,\P)$, i.e. $(t,\o) \mapsto \theta_t(\o)$ is $\mcB(\R) \otimes \mcF / \mcF$--measurable, $\theta_0 =$ id, $\theta_{t+s} = \theta_t \circ \theta_s$ and $\theta_t$ is $\P$-preserving, for all $s,t \in \R$.

\begin{definition}
  Let $(H,d)$ be a complete separable metric space. A random dynamical system (RDS) over $\theta_t$ is a measurable map
  \begin{align*}
    \vp: \R_+ \times H \times \O &\to H \\
    (t,x,\o)                     &\mapsto \vp(t,\o)x
  \end{align*}
  such that $\vp(0,\o) =$ id and $\vp$ satisfies the cocycle property, i.e. 
  \[ \vp(t+s,\o) = \vp(t,\theta_s \o) \circ \vp(s,\o), \]
  for all $t,s \in \R_+$ and all $\o \in \O$. $\vp$ is said to be a continuous RDS if $\P$-a.s. $x \mapsto \vp(t,\o)x$ is continuous for all $t \in \R_+$.
\end{definition}

With the notion of an RDS at our disposal we can now recall the stochastic generalization of notions of absorption, attraction and $\O$-limit sets.

\begin{definition}\label{def:rds_basics}
Let $(H,d)$ be as in Definition 1.7
\begin{enumerate}[(i.)]
  \item A set-valued map $K: \O \to 2^H$ is called measurable if for all $x \in H$ the map $\o \mapsto d(x,K(\o))$ is measurable, where for nonempty sets $A,B \in 2^H$ we set $d(A,B)=\sup\limits_{x \in A} \inf\limits_{y \in B} d(x,y)$ and $d(x,B) = d(\{x\},B)$. A measurable set-valued map is also called a random set.
  \item Let $A$, $B$ be random sets. $A$ is said to absorb $B$ if $\P$-a.s. there exists an absorption time $t_B(\o) \geq 0$ such that for all $t \ge t_B(\o)$
        \[ \vp(t,\theta_{-t}\o)B(\theta_{-t}\o) \subseteq A(\o).\]
        $A$ is said to attract $B$ if 
        \[ d(\vp(t,\theta_{-t}\o)B(\theta_{-t}\o),A(\o)) \xrightarrow[t \to \infty]{} 0,\ \P\text{-a.s. }.\]
  \item For a random set $A$ we define the $\O$-limit set to be
        \[ \O_A(\o) := \bigcap_{T \ge 0} \overline{ \bigcup_{t \ge T} \vp(t,\theta_{-t}\o)A(\theta_{-t}\o)}. \]
  \end{enumerate}
\end{definition}

\begin{definition}
  A random attractor for an RDS $\vp$ is a compact random set $A$ satisfying $\P$-a.s.
  \begin{enumerate}
   \item $A$ is invariant, i.e. $\vp(t,\o)A(\o)= A(\theta_t \o)$ for all $t > 0$.
   \item $A$ attracts all deterministic bounded sets $B \subseteq H$.
  \end{enumerate}
\end{definition}

The following proposition yields a sufficient criterion for the existence of a random attractor of an RDS $\vp$. We also refer to \cite{KS04} for a different approach based on ergodicity, which however seems less suitable for our case.
\begin{proposition}[cf. \cite{cf1}, Theorem 3.11]\label{prop:suff_criterion}
  Let $\vp$ be an RDS and assume the existence of a compact random set $K$ absorbing every deterministic bounded set $B \subseteq H$. Then there exists a random attractor $A$, given by
  \[ A(\o) = \overline{ \bigcup_{B \subseteq H,\ B \text{ bounded}} \O_B(\o)}.\]
\end{proposition}
From now on we take $H:=H_0^1(\L)^*$ with metric determined by its norm $ \| \cdot \|_H$.
Since we aim to apply Proposition \ref{prop:suff_criterion} to prove the existence of a random attractor for \eqref{e1.0}, we first need to define the RDS  associated to \eqref{e1.0}. We take $(\O,\mcF,\P)$ to be the canonical two-sided Wiener space, i.e. $\O = C_0(\R,\R^m)$ and $\theta_t$ to be the Wiener shift given by $\theta_t \o := \o(t+\cdot)-\o(t)$. As in \cite[pp. 375--377]{cf1} we consider $Y(t,s,x) := X(t,s,x) - QW_t $. Then we have for all $s \in \R, x \in H, \P$-a.s.:
  \[ Y(t,s,x) = x-QW_s + \int_s^t \D \Phi(Y(r,s,x)+QW_r)dr,\ \forall t \ge s. \]
We can rewrite this as an $\o$-wise equation:
\begin{equation}\label{eqn:o-wise}
  Z_t(\o) = x-QW_s(\o) + \int_s^t A_\o(r,Z_r(\o)) dr,\ \forall t \ge s,
\end{equation}
where $A_\o(r,v) := \D \Phi(v + QW_r(\o))$. Since for each fixed $\o \in \O$, $A_\o: V \to V^*$ is hemicontinuous, monotone, coercive and bounded we can apply \cite[Theorem 4.2.4]{pr} to obtain the unique existence of a solution 
\begin{equation}\label{eqn:lp-bdd}
  Z(t,s,x,\o) \in L^{p+1}_{loc}([s,\infty);V) \cap C([s,\infty),H)
\end{equation}
 to \eqref{eqn:o-wise} for all $x \in H$, $\o \in \O$, $s \in \R$ and its continuous dependence on the initial condition $x$. We now define in analogy to \cite{cf2} 
\begin{align}\label{eqn:def_rds}
  S(t,s,\o)x &:= Z(t,s,x,\o) + QW_t(\o),\ s,t \in \R;\ s \le t \nonumber\\
  \vp(t,\o)x &:= S(t,0,\o)x = Z(t,0,x,\o) + QW_t(\o),\ t \ge 0.
\end{align}
By uniqueness for \eqref{e1.0} $S(t,s,\o)x$ is a version of $X(t,s,x)(\o)$, for each $x \in H$, $s \in \R$. For fixed $s,\o,x$ we at times abbreviate $S(t,s,\o)x$ by $S_t$ and $Z(t,s,x,\o)$ by $Z_t$. By the pathwise uniqueness of the solution to equation \eqref{eqn:o-wise} we have for all $\o \in \O, \; r,s,t \in \R, \; s \leq r \leq t,$
\begin{align}\label{1.6'}
  S(t,s,\o) &= S(t,r,\o)S(r,s,\o) \tag{1.8'}\\
  S(t,s,\o) &= S(t-s,0,\theta_s\o). \label{1.6''} \tag{1.8''}
\end{align}
The joint measurability of $\vp: \R_+ \times H \times \O \to H$ follows from the construction of the solutions $Z_t$ (for details we refer to \cite{GLR10}). Hence $\vp$ defines an RDS. We can thus apply Proposition \ref{prop:suff_criterion} to prove the existence of a random attractor for $\vp$. For this we need to prove the existence of a compact set $K(\o)$, which absorbs every bounded deterministic set in $H$, $\P$-almost surely. This set will be chosen as $K(\o) := \overline{ B_{L^2} (0, \kappa (\o ) )}^H$, where $B_{L^2} (0, \kappa)$ denotes the ball with center $0$ and radius $\kappa$ in $L^2$. Note that since $\vp(t,\theta_{-t}\o) = S(t,0,\theta_{-t}\o) = S(0,-t,\o)$, this amounts to proving pathwise bounds on $S(0,-t,\o)x$ in the $L^2$-norm, where we use the compactness of the embedding $L^2(\L) \hookrightarrow H$. In order to get such estimates we consider norms $\norm{\cdot}_{H_a}$ on $H$ such that $\norm{\cdot}_{H_a}\uparrow \norm \cdot _{L^2}$, as $a \downarrow 0$. These are defined as the dual norms (via the Riesz isomorphism) of the norms
  \[ H^1_0(\Lambda) \owns u \mapsto \left(a \int_\Lambda |\nabla u|^2 d\xi + \int u^2  d\xi\right)^{1/2}.\]
Then for $s \le t$ we have (see e.g. \cite[Theorem 2.6 and Lemma 2.7 (i),(ii)]{rw}) for $a:=\frac{1}{n}$
\begin{equation}\label{eqn:approx-ito}\tag{1.9}
 \norm{ Z_t}^2_{H_\frac{1}{n}} \, = \, \norm {Z_s }^2_{H_\frac{1}{n}} + 2 \int^t_s \langle \Phi(S_r), n(1-\frac{1}{n} \Delta)^{-1} Z_r -nZ_r \rangle dr,
\end{equation}
where for $f,g : \Lambda \to \R$ measurable we set
 \[\langle f,g \rangle := \int_\L\, f\, g\, d\xi\]
if $|fg| \in L^1(\L)$. We shall use \eqref{eqn:approx-ito} in a crucial way several times below.


\section{Estimates for $\norm{S_t}_{H}$ and bounded absorption}
In this section we construct bounded random sets that absorb trajectories with respect to the weak norm $\|\cdot\|_H$. For these estimates only the standard assumptions (A1)-(A3) are needed.

\begin{theorem}\label{thm:H-bound}
Let $\b \in (0,\infty)$, with $\b \leq \frac{a}{2}$, if $p=1$. Then there exists a function $p_1^{(\b)}: \R\times \O\to \R_+$ such that $p_1^{(\b)}(t,\cdot)$ is $\F$--measurable for each $t \in \R$, $p_1^{(\b)}(\cdot,\o$) is continuous and of at most polynomial growth for $|t| \to \infty, \o \in \O$ and for all $x \in H$, $\o \in \O_0$ (cf. Remark \ref{1.1} (i)), $s \in \R$:
\begin{align}\label{eqn:H-bound}\tag{2.1}
  \norm{Z(t_2,s,x,\o)}_H^2 \leq \norm{Z(t_1,s,x,\o)}_{H}^2 -\b \int^{t_2}_{t_1} \norm{Z(r,s,x,\o)}^2_2 dr + \int^{t_2}_{t_1} p_1^{(\b)} (r,\o) dr, \\ \text{for all } s \le t_1 \le t_2.\nonumber
\end{align}
\end{theorem}
\begin{proof}
We fix $x,\o,s$ and set $Z_r:=Z(r,s,x,\o), \; S_r := S(r,s,\o)x$ for $r \geq s$. All constants appearing in the proof below are, however, independent of $x,\o$ and $s$!\\
Since for $s \le t_1 \le t_2$
    \[ \norm{Z_{t_2} }^2 _H = \norm{Z_{t_1}}^2_H - 2 \int^{t_2}_{t_1} \langle Z_r , \Phi(S_r) \rangle dr, \]
  we have for $dr$--a.e. $r \in [s,\infty)$ by (A2)
  \begin{align*}
    \frac d{dr} \norm{Z_r} ^2 _H &= - 2 \langle Z_r , \Phi(S_r) \rangle \\
    & =-  2 \langle S_r - QW_r , \Phi(S_r) \rangle \\
    & =-  2 \langle S_r , \Phi(S_r) \rangle + 2 \langle QW_r , \Phi(S_r) \rangle \\
    & \le -  2a \int _\L |S_r|^{p +1} d\xi  + 2 \int_\L \left( |QW_r \Phi(S_r)| + c \right) d\xi .
  \end{align*}
  By Young's inequality, for arbitrary $\epsilon >0$ and some $C_\epsilon (=C_\epsilon (p)), C_1, C_2 \in \R$ we have by (A3)
  \begin{align*}
    \int_\L |QW_r \Phi(S_r)| d\xi & 
      \le \int_\L \left( C_\epsilon |QW_r|^{p+1} + \epsilon |\Phi(S_r)|^{\frac{p+1}{p}} \right) d\xi \\
      &\le \epsilon C_1 \|S_r\|_{p+1}^{p+1} + C_\epsilon \|QW_r\|_{p+1}^{p+1} + \epsilon C_2 |\L|,
  \end{align*}
  where $|\L|:= \int_\L \; d\xi .$
  Thus by choosing $\epsilon = \frac{a}{C_1}$ we obtain for $dr$--a.e. $r \in [t_1,t_2]$
  \begin{align*}
      \frac d{dr} \norm{Z_r} _H^2 
	&\le - a\|S_r\|_{p+1}^{p+1} + C_\epsilon \|QW_r\|_{p+1}^{p+1}+2|\L|(c+C_3).
  \end{align*}
where $C_3 := \frac{aC_2}{C_1}$.\\
  Now, if $p > 1$, then for each $\b > 0$ we can find a $C_{\b}$ such that for all $y \in \R$ one has $a|y|^{p+1} \ge 2\b|y|^2 - C_{\b}$. If $p=1$, then we have the same, provided $\b \in (0, \frac{a}{2}]$. We obtain
  \[ a\|S_r\|_{p+1}^{p+1} \ge 2\b \|S_r\|_2^2 - |\L|C_{\b} = 2\b\|Z_r + QW_r\|_2^2 - |\L|C_{\b} \ge \b \|Z_r\|_2^2 - 2\b \|QW_r\|_2^2 - |\L|C_{\b}. \]
  Hence for
  $$ p_1^{(\b)}(r,\o):= 
    \begin{cases}
        2\b \|QW_r\|_2^2 + |\L|C_{\b}  + C_\epsilon \|QW_r\|_{p+1}^{p+1}+2|\L|(c+C_3) & \text{, if } \o \in \O_0 \\
        0 & \text{, otherwise }
    \end{cases} $$
  we obtain for $dr$-a.e. $r \in [t_1,t_2]$
  \begin{align*}
      \frac d{dr} \norm{Z_r} _H^2 \leq - \b \|Z_r\|_2^2 + p_1^{(\b)} (r,\o).
  \end{align*}
and the assertion follows.
\end{proof}

\begin{corollary}\label{cor:H-bound}
  Let $\b \in (0,\infty)$, with $\b \leq \frac{a}{2}$ if $p=1$ and let $t \in \R$. Then there exists an $\F$--measurable function $q_1^{(\b,t)}: \O\to \R,$
  such that for all $x \in H$, $\omega \in \Omega_0$ and $s \le t$
  \begin{equation}\label{star}\tag{2.2}
    \norm{Z(t,s,x,\o)}_H^2 \leq q_1^{(\b,t)}(\o) + e^{-\frac{\b}{e^2} (t-s)} \norm{Z(s,s,x,\o)} _H^2	.
  \end{equation}
\end{corollary}
\begin{proof}
  Since the embedding $L^2 \hookrightarrow H$ is continuous, there is a constant $c> 0$ such that $\norm{v} _H\leq c\norm{v}_{2}$, for all $v \in L^2$. Hence by Theorem 2.1
    \[\frac d {dr} (\norm{Z_r}_H^2) \leq - \frac{\b}{c^2} \norm{Z_r} _H^2 + p_1^{(\b)}(r,\o) \ \ dr\text{--a.e. on } [s,t].\]
  Hence by Gronwall's Lemma the assertion follows with $q_1^{\b,t}(\o):=\int^t_{-\infty} e^{-\frac{\b}{c^2}(t-r)}p_1(r,\o)dr$.
\end{proof}

\begin{corollary}[Bounded absorption]\label{cor:bdd-absorption}
  Let $t \in \R$. Then there is an $\F$--measurable function $q_1^{(t)}:\O \to \R$ such that for each $\varrho >0$ there is an $s(\varrho) \le t$ such that for all $\o \in \O_0$, $x \in H$ with $\norm{x}_H \leq \varrho$
    \[Z(t,s,x,\o) \in \bar {B}_H(0,q_1^{(t)}(\o)), \quad \text{for all } s \leq s(\varrho)\]
  i.e. there exists a bounded random set absorbing $(Z_t)$ at time $t$.
\end{corollary}
\begin{proof}
  Let $\b := \frac{a}{2}$. By Corollary \ref{cor:H-bound}, we have for $\tilde{\b} := \frac{\b}{c^2}$
  \begin{align*}
    \norm{Z_t}_H^2 
      &\leq e^{-\tilde{\b} (t-s)}\norm{Z_s} _H^2 + q_1^{(\b,t)}\\
      \leq & 2 e^{-\tilde{\b} (t-s)}(\norm{x}_H^2 + \norm{QW_s}_H^2) + q_1^{(\b,t)}\\
      \leq &2 \varrho ^2 e^{-\b (t-s)}+ 2 e^{-\tilde{\b} (t-s)}\norm {QW_s}_{H}^2 + q_1^{(\b,t)},
  \end{align*}
  for all $t \ge s.$ Hence the result follows with
    \[q_1^{(t)} := 1+q_1 ^{(\b,t)} + 2 \sup_{s \leq t}(e^{-\tilde{\b} (t-s)}\norm{QW_s}_{H}^2)\]
  and $s(\varrho ) \le t$ chosen so that $2 \varrho ^2 e^{-\tilde{\b} (t-s)} \le 1$ for all $s \le s(\varrho)$.
\end{proof}

We will need the following auxiliary estimate.

\begin{corollary}\label{cor:aux_estimate}
  There is an $\F$--measurable function $q : \O \to \R_+$ such that for each $\varrho >0$ there exists
  $ s(\varrho ) \leq -1 $ such that for all $\o \in \O_0, \; x\in H$ with $\norm{x}_H\leq \varrho$
  \begin{equation*}
     \int_{-1}^0 \norm { S(r,s,\o)x }_{2}^2 dr \leq q(\o) \text{ for all }s \leq s(\varrho).
  \end{equation*}
\end{corollary}
\begin{proof}
  Using \eqref{eqn:H-bound} in Theorem 2.1 with $t_1=-1,t_2=0$ and then using Corollary \ref{cor:bdd-absorption} for $t=-1$ yields for $\b = \frac{a}{2}$ and $s \leq s(\varrho)$, where $s(\varrho) \leq -1$ is as in Corollary 2.3,
  \begin{align*}
    \b \int_{-1}^{0} \norm {S(r,s,\o)x}_{2}^2 dr 
      &\leq 2 \norm{Z(-1,s,x,\o)}_H^2 + 2 \int_{-1}^{0} p_1^{(\b)}(r,\o) dr + 2\b \int^0_{-1} \norm{QW_r(\o)}^2_2 dr \\
      &\leq \b q (\o),
  \end{align*}
where $q(\o):= \frac{2}{\b} q_1^{(-1)} (\o) + \frac{2}{\b} \int_{-1}^0 p_1^{(\b)} (r,\o) dr + 2 \int_{-1}^0 \norm{QW_r(\o)}^2_2 dr.$
\end{proof}


\section{Estimate for $\norm{S_t}_{2}$ and compact absorption}
In this section we utilize the stronger assumptions from Hypothesis \ref{phi_conditions_2} and \ref{phi_conditions_3} in order to obtain absorption with respect to the strong norm $\|\cdot\|_2$. By the compact embedding $L^2(\L) \hookrightarrow H$ this will prove the required compactness properties.
\begin{theorem}\label{thm:L2-bound} 
 Suppose that either Hypothesis \ref{phi_conditions_2} or Hypothesis \ref{phi_conditions_3} holds. Let $\a> 0$, with $\a \in (0, \frac{\a \lambda_1}{2}]$ if $p=1$. Then there is a function $p_2^{(\a)}: \R\times \O \to \R$ such that $p_2^{(\a)}(t,\cdot)$ is $\F$--measurable for each $t \in \R$, $p_2^{(\a)}(\cdot,\o$) is continuous and of at most polynomial growth for $|t| \to \infty,\ \o \in \O$ and for all  $x \in L^2 (\L), \o \in \O_0$ (cf. Remark \ref{1.1} (i)), $s \in \R$:
  \begin{align}\label{eqn:L2-bound}
     \norm {Z(t_2,s,x,\o)}_{2}^2 \leq \norm {Z(t_1,s,x,\o)} _{2}^2 -\a \int^{t_2}_{t_1} \norm{Z(r,s,x,\o)}^2_2 \; dr + \int^{t_2}_{t_1} p _2^{(\a)}(r, \o) dr \tag{3.1} \\ \text{for all } s \le t_1 \le t_2. \nonumber
  \end{align}
  In particular, $t \to Z_t$ is strongly right continuous in $L^2(\L)$.
\end{theorem}
\begin{proof}
Again we fix $x,\o,s$ and use the abbreviation $Z_r:= Z(r,s,x,\o), \ S_r := S(r,s,\o)x$ for $r \in [s,\infty)$. But all constants appearing in the proof below are independent of $x,\o$ and $s$.\\
  \textbf{Case 1}: Assume Hypothesis \ref{phi_conditions_2}.

  Let $t_1 \ge s$ such that $Z_{t_1} \in L^2(\L)$ and $t_2 \ge t_1$. \eqref{eqn:approx-ito} implies
  \begin{align}\label{eqn_L^2_1}
    \norm {Z_{t_2}}^2_{H_{\frac1n}} = \norm{Z_{t_1}}^2_{H_{\frac1n}} 
    &+ 2 \int^{t_2}_{t_1} \langle \Phi(S_r),n(1-\tfrac1n \Delta)^{-1} S_r -n S_r\rangle dr \tag{3.2} \\
    & - 2 \int^{t_2}_{t_1} \langle \Phi(S_r), \Delta(1-\tfrac1n\Delta)^{-1}Q W_r\rangle dr.\nonumber
   \end{align}

  We now recall a formula given in \cite[Lemma 5.1 (ii)]{rw}. Let $q,q' >1$ such that $\frac{1}{q}+\frac{1}{q'}=1$, $f \in L^q(\L)$, $g \in L^{q'}(\L)$ and $p_n(\xi,d\tilde{\xi})$ the kernel corresponding to $(1-\tfrac1n \Delta)^{-1}$ (cf. \cite[Lemma 5.1 (i)]{rw}). Using the symmetry of $(1-\tfrac1n \Delta)^{-1}$ in $L^2(\L)$ we obtain
  \begin{align*}
    \langle f, g-(1-\tfrac1n \Delta)^{-1}g \rangle
    &= \frac{1}{2} \int_\L \int_\L (f(\tilde{\xi})-f(\xi))(g(\tilde{\xi})-g(\xi))p_n(\xi,d\tilde{\xi}) d\xi \\
      &\hskip13 pt + \int_\L (1-(1-\tfrac1n \Delta)^{-1}1)fg d\xi.
  \end{align*}
  Using this and proceeding analogously to the calculation following formula (5.6) in \cite{rw} yields for $dr$-a.e. $r \in [s,\infty)$
  \begin{align*}
    \langle \Phi(S_r),n(1-\tfrac1n \Delta)^{-1} S_r -n S_r \rangle 
      &= -n \langle \Phi(S_r), S_r - (1-\tfrac1n \Delta)^{-1} S_r\rangle \nonumber\\
      &= -\frac{n}{2} \int_\L \int_\L [ \Phi(S_r(\tilde{\xi}))-\Phi(S_r(\xi)) ] [ S_r(\tilde{\xi})- S_r(\xi)]p_n(\xi,d\tilde{\xi})d\xi \nonumber\\
      &\ \ \ \ -n \int_\L(1-(1-\frac{1}{n}\Delta)^{-1}1)\Phi(S_r)S_r d\xi \nonumber\\
      &\le -\frac{n}{2} \int_\L \int_\L (\zeta(S_r(\tilde{\xi}))-\zeta(S_r(\xi)))^2 p_n(\xi,d\tilde{\xi})d\xi \\ 
      &\ \ \ \ -n \int_\L(1-(1-\frac{1}{n}\Delta)^{-1})\zeta(S_r)^2 d\xi \nonumber\\ 
      &= -n \langle \zeta(S_r), (1-(1-\frac{1}{n}\Delta)^{-1}) \zeta(S_r)  \rangle \nonumber\\
      &= - \mathcal{E}^{(n)}(\zeta(S_r),\zeta(S_r)), \nonumber
  \end{align*}
    where $(\df^{(n)}, \dom(\df^{(n)}))$ is the closed coercive form on $L^2(\L)$ with $\dom(\df^{(n)}) = H_0^1(\L)$ and generator $n(1-(1-\frac{1}{n}\Delta)^{-1})=\D(1-\frac{1}{n}\Delta)^{-1}$. We obtain from \eqref{eqn_L^2_1}:
  \begin{align}\label{eqn_L^2_3}
    &\norm {Z_{t_2}}^2_{H_{\frac1n}} + 2\int^{t_2}_{t_1}\mathcal{E}^{(n)}(\zeta(S_r),\zeta(S_r))dr \tag{3.3} \\
    &\le \norm{Z_{t_1}}^2_{H_{\frac1n}} - 2 \int^{t_2}_{t_1} \langle \Phi(S_r), \Delta(1-\tfrac1n\Delta)^{-1}Q W_r\rangle dr.\nonumber
  \end{align}

  Next we prove an upper bound for the second term on the right hand side of \eqref{eqn_L^2_3}. Note that we shall make use of the assumption 
  $QW_t \in C^2_0$ here. Using Young's inequality and (A3), for all $\epsilon >0$ and some $C_\epsilon, C_1, C_2 >0$ we obtain for $dr$--a.e. $r \in [s,\infty)$

  \begin{align*}
    |\langle \Phi(S_r), \Delta(1-\frac{1}{n}\Delta)^{-1}QW_r \rangle|
      &= |\langle \Phi(S_r), (1-\frac{1}{n}\Delta)^{-1} \Delta QW_r \rangle| \\
      &\le \epsilon \int_\L |\Phi(S_r)|^{\frac{p+1}{p}} d\xi +  C_\epsilon \int_\L |((1-\frac{\Delta}{n})^{-1} \Delta QW_r )|^{p+1} d\xi \\
      &\le \epsilon C_1 \|S_r\|_{p+1}^{p+1} + C_\epsilon \|\Delta QW_r\|_{p+1}^{p+1} + C_2 .
  \end{align*}

  Hence
  \begin{align}\label{eqn_L^2_4}
    &\norm {Z_{t_2}}^2_{H_{\frac1n}} + 2\int^{t_2}_{t_1}\mathcal{E}^{(n)}(\zeta(S_r),\zeta(S_r))dr \tag{3.4} \\
    & \le \norm{Z_{t_1}}^2_2 + 2 \int^{t_2}_{t_1} \left[ \epsilon C_1 \|S_r\|_{p+1}^{p+1} + C_\epsilon \|\Delta QW_r\|_{p+1}^{p+1} + C_2 \right] dr < \infty.\nonumber
  \end{align}
We note that by \eqref{eqn:lp-bdd} the right hand side of \eqref{eqn_L^2_4} is indeed finite. Since $\df^{(n)} (\zeta(S_r),\zeta(S_r))$ is increasing in $n$, we conclude that $\sup\limits_{n \in \N} \mathcal{E}^{(n)}(\zeta(S_r),\zeta(S_r)) < \infty$ for d$r$-a.e. $r \in [t_1,\infty)$. By (A2)' and (A3) we know that for some $c_1, c_2 \ge 0$
  \begin{equation*}
    \zeta(s)^2 \le \Phi(s)s \le c_1 |s|^{p+1} + c_2 |s|.
  \end{equation*}
  Since $S_r \in L^{p+1}(\L)$ this implies $\zeta(S_r) \in L^2(\L)$ for d$r$-a.e. $r \in [t_1,\infty)$. We now recall the following result from the theory of Dirichlet forms: Let $(\df, \dom(\df))$ be the closed coercive form on $L^2(\L)$ given by $\df(f,g) = \int_\L \langle \nabla f, \nabla g \rangle_{\R^d}d\xi$ for $f,g \in \dom(\df) = H_0^1(\L)$. From \cite[Chap. I, Theorem 2.13]{rma} we know for $f \in L^2(\L)$, that $f \in \dom(\df) = H_0^1(\L)$ iff $\sup\limits_{n \in \N} \df^{(n)}(f,f) < \infty$. Moreover, $\lim\limits_{n \rightarrow \infty} \df^{(n)}(f,g) = \df(f,g) = \int_{\L} \langle \nabla f, \nabla g \rangle_{\R^d}d\xi$ for all $f,g \in \dom(\df)$. Hence we obtain for d$r$-a.e. $r \in [t_1,\infty)$ that $\zeta(S_r) \in \dom(\df) = H_0^1(\L)$ and that
  \begin{equation*}
    \lim_{n \rightarrow \infty} \df^{(n)}(\zeta(S_r),\zeta(S_r)) = \df(\zeta(S_r),\zeta(S_r)) = \int_\L |\nabla \zeta(S_r)|_{\R^d}^2 d\xi. 
  \end{equation*}
  Using Fatou's lemma and taking $n \rightarrow \infty$ in \eqref{eqn_L^2_4} yields 
  \begin{align}\label{2.6}
    &\norm{Z_{t_2}}^2_{2} + 2 \int^{t_2}_{t_1} \int_\L |\nabla \zeta(S_r)|_{\R^d}^2 d\xi \ dr \tag{3.5} \\
    &\le \norm{Z_{t_1}}^2_{2} + 2\epsilon C_1 \int^{t_2}_{t_1} \|S_r\|_{p+1}^{p+1} dr + \int^{t_2}_{t_1} (C_\epsilon \|\Delta QW_r\|_{p+1}^{p+1} + C_2) dr\nonumber.
  \end{align}
  Since $Z_s = x-QW_s \in L^2(\L)$, for all $t_1 \ge s$ we obtain $Z_{t_1} \in L^2(\L)$ and thus \eqref{2.6} holds for all $t_2 \ge t_1 \ge s$. 

Choosing $\epsilon = \frac{a \lambda_1 }{2C_1}$, applying Poincar\'e's inequality and using the fact that if $p>1$ for each $\a > 0$ we can find
$\tilde{C_\a}\ge 0$ such that for all $y \in \R$ one has $a \l_1 |y|^{p+1} \ge 2\a |y|^2 - \tilde{C}_\a$; the same is true for $p=1$, if $\a \in (0, \frac{\a \lambda_1}{2}]$. We obtain from (A2)' that
\begin{align*}
  \norm{Z_{t_2} }^2_{2}
    &\le \norm{Z_{t_1}}^2_{2} - 2 \lambda_1 \int^{t_2}_{t_1} \|\zeta(S_r)\|_{2}^2 dr + a\lambda_1 \int^{t_2}_{t_1} \|S_r\|_{p+1}^{p+1} dr + \int^{t_2}_{t_1} (C_\epsilon\|\Delta QW_r\|_{p+1}^{p+1} + C_2) dr \\
    &\le \norm{Z_{t_1}}^2_{2} - a \lambda_1 \int^{t_2}_{t_1} \|S_r\|^{p+1}_{p+1}  dr + \int^{t_2}_{t_1} (C_\epsilon\|\Delta QW_r\|_{p+1}^{p+1} + C_2 + c) dr \\
    &\le \norm{Z_{t_1}}^2_{2} - 2\a \int^{t_2}_{t_1} \|S_r\|^2_2  dr + \int^{t_2}_{t_1} (C_\epsilon\|\Delta QW_r\|_{p+1}^{p+1} + C_2 + c + \tilde{C}_\a) dr.
\end{align*}
Now
\[\norm{Z_r}^2_{2} = \norm{S_r - QW_r }^2_{2} \leq 2\left( \norm {S_r}^2_{2} + \norm{QW_r}^2_{2}\right), \]
whence
\begin{align}\label{eqn_final_bound}
  \norm{Z_{t_2} }^2_{2} &\le \norm{Z_{t_1}}^2_{2} - \alpha \int^{t_2}_{t_1} \|Z_r\|^2_2  dr + \int^{t_2}_{t_1} p_2^\alpha(r,\omega) dr, \tag{3.6}
\end{align}
for $\alpha >0$ arbitrary and 
$$
  p_2^{(\a)}(r,\o) :=
  \begin{cases}
    C_\epsilon\|\Delta QW_r\|_{p+1}^{p+1} + C_2 + c + \tilde{C}_\a +2\a \norm{ QW_r}^2 _{2} & \text{, if } \o \in \O_0 \\
    0 & \text{, else. }
  \end{cases}
$$

To obtain right continuity of $Z_t$ in $L^2(\L)$ first note that by \eqref{2.6} applied for $t_1 = s$ and continuity of $Z_t$ in $H$ we obtain weak continuity in $L^2(\L)$. Now for $t_n \downarrow t$ by  \eqref{2.6} applied to $t_1 = t$ we obtain 
  \[  \limsup_{n \rightarrow \infty} \norm{Z_{t_n}}^2_2 \le \norm{Z_t}^2_2; \]
which implies the right continuity of $Z_t$ in $L^2(\L)$.

\vskip25pt
\textbf{Case 2}: Assume Hypothesis \ref{phi_conditions_3}.

Let $\zeta$ be as defined in Lemma \ref{lemma:hyp1.4->hyp1.1} and again let $t_1 \ge s$ such that $Z_{t_1} \in L^2(\L)$ and $t_2 \ge t_1$. In order to prove \eqref{eqn:L2-bound} in the case $QW_t \in C^1_0(\L)$ we need to be more careful when bounding the second term on the right hand side of \eqref{eqn_L^2_3}. For this we need the regularity result proved in Lemma \ref{lemma:regularity} below, which implies that for every $\epsilon > 0$ there exist constants $C_\epsilon, \tilde{C}_\epsilon(=C_\epsilon (p),\tilde{C}_\epsilon(p))$ such that for d$r$-a.e. $r \in [s,\infty)$
\begin{align}\label{3.6'}
  - \langle \Phi(S_r), \Delta(1-\frac{1}{n}\Delta)^{-1}QW_r \rangle \nonumber
    &= \langle \nabla \Phi(S_r), \nabla(1-\frac{1}{n}\Delta)^{-1} QW_r \rangle \nonumber\\
    &\le \epsilon\|\nabla \Phi(S_r)\|_\frac{p+1}{p}^\frac{p+1}{p} + C_\epsilon \|\nabla(1-\frac{1}{n}\Delta)^{-1} QW_r\|_{p+1}^{p+1} \tag{3.6'}\\
    &\le \epsilon\|\nabla \Phi(S_r)\|_\frac{p+1}{p}^\frac{p+1}{p} + \tilde{C}_\epsilon \|\nabla QW_r\|_{p+1}^{p+1}. \nonumber
\end{align}
Now using Lemma \ref{lemma:regularity} and \eqref{3.6'} with $\ve = 1$ in \eqref{eqn_L^2_3} yields for some constants $c,C \in \R$
\begin{align*}
   & \norm {Z_{t_2}}^2_{H_{\frac1n}} + 2\int^{t_2}_{t_1}\mathcal{E}^{(n)}(\zeta(S_r),\zeta(S_r))dr \\
   & \le \norm{Z_{t_1}}^2_{H_{\frac1n}} + 2 \int^{t_2}_{t_1} \left[ \|\nabla \Phi(S_r)\|_\frac{p+1}{p}^\frac{p+1}{p} + \tilde{C}_1 \|\nabla QW_r\|_{p+1}^{p+1}\right] dr \\
   & \le c\norm{Z_{t_1}}^2_2 + C \int^{t_2}_{t_1} (\|\nabla QW_r\|_{p+1}^{p+1}+1) dr < \infty\ .
\end{align*}
Now we can proceed as after \eqref{eqn_L^2_4} to deduce $\zeta(S_r) \in \dom(\df) = H_0^1(\L)$ and
  \begin{equation*}
    \lim_{n \rightarrow \infty} \df^{(n)}(\zeta(S_r),\zeta(S_r)) = \int_\L |\nabla \zeta(S_r)|_{\R^d}^2 d\xi,
  \end{equation*}
for d$r$-a.e. $r \in [s,\infty)$.
Since $\Phi'(r) > 0$, $\zeta(s)=\int_0^{s} \sqrt{\Phi'(r)} dr$ is $C^1(\R)$ with continuous inverse $\zeta^{-1}$. Thus
\begin{align*}
  \Phi(x) &= \int_0^x \Phi'(r) dr = \int_0^x \sqrt{\Phi'(r)} \sqrt{\Phi'(r)}dr \\
          &= \int_0^x \zeta'(r) \sqrt{\Phi'(r)}dr = \int_0^{\zeta(x)} \sqrt{\Phi'(\zeta^{-1}(r))}dr = F(\zeta(x)),
\end{align*}
where $F(s) := \int_0^{s} \sqrt{\Phi'(\zeta^{-1}(r))}dr$. Since $F \in C^1(\R)$, $\zeta(S_r) \in H_0^1(\L)$ for d$r$-a.e. $r \in [s,\infty)$ and 
$F'(\zeta(S_r))\nabla \zeta(S_r) = \sqrt{\Phi'(S_r)}\nabla \zeta(S_r) \in L^1(\L)$ (by (1.4)), we have $\Phi(S_r) = F(\zeta(S_r)) 
\in H^{1,1}_0(\L)$ for d$r$-a.e. $r \in [s,\infty)$ with
\begin{equation}\label{2.9}
  \nabla \Phi(S_r) = \sqrt{\Phi'(S_r)}\nabla \zeta(S_r) \in L^1(\L).\tag{3.7}
\end{equation}
  By (A2)' and \eqref{1.5} there are some constants $C_1,C_2$ such that
    \[ \zeta'(r)^{2\frac{p+1}{p-1}} \le C_1 \zeta(r)^2 + C_2. \]
  Using \eqref{2.9} and then Young's and Poincar\'e's inequalities, for some constant $C$ (which may change from line to line) we have for d$r$-a.e. $r \in [s,\infty)$
\begin{align}\label{eqn:bound_phi_zeta}
  \|\nabla \Phi(S_r)\|_{\frac{p+1}{p}}^{\frac{p+1}{p}}
  &= \int_\L |\nabla\Phi(S_r)|^\frac{p+1}{p} d\xi = \int_\L |\sqrt{\Phi'(S_r)} \nabla \zeta(S_r)|^\frac{p+1}{p} d\xi \nonumber \\
  &= \int_\L |\zeta'(S_r) \nabla \zeta(S_r)|^\frac{p+1}{p} d\xi \le \|\nabla \zeta(S_r)\|_2^2 + C \int_\L |\zeta'(S_r)|^{2\frac{p+1}{p-1}} d\xi \\
  &\le \|\nabla \zeta(S_r)\|_2^2 + C \|\zeta(S_r)\|_2^2 + C \le C \|\nabla \zeta(S_r)\|_2^2 + C. \nonumber
\end{align}
We can now go on with bounding the second term on the right hand side of \eqref{eqn_L^2_3} as follows:\\
\eqref{3.6'} and \eqref{eqn:bound_phi_zeta} imply that for d$r$-a.e. $r \in [s,\infty)$
\begin{align}\label{eqn:noise-bound}
  \langle \Phi(S_r), \Delta(1-\frac{1}{n}\Delta)^{-1}QW_r \rangle \nonumber
    &\le \ve\|\nabla \Phi(S_r)\|_\frac{p+1}{p}^\frac{p+1}{p} + \tilde{C}_\ve\|\nabla QW_r\|_{p+1}^{p+1} \\
    &\le \ve C_1 \|\nabla \zeta(S_r)\|_2^2 + \ve C_2 + \tilde{C}_\ve\|\nabla QW_r\|_{p+1}^{p+1}.
\end{align}
Using this with $\ve = \frac{1}{C_1}$ in \eqref{eqn_L^2_3} and letting $n \to \infty$ yields for some constant $C$
  \begin{align}
    \norm {Z_{t_2}}^2_2 + \int^{t_2}_{t_1}\|\nabla \zeta(S_r)\|_2^2 dr \le \norm{Z_{t_1}}^2_2 +  2C \int^{t_2}_{t_1} (1 + \|\nabla QW_r\|_{p+1}^{p+1}) dr.
  \end{align}
Now we can proceed as done in the proof of Case 1 after \eqref{2.6}.
\end{proof}

\begin{remark}
  As indicated before the arguments in the proof can easily be generalized to noise $QW_t \in H_0^{2,p+1}(\L)$ $(QW_t \in H_0^{1,p+1}(\L)$ resp.$)$.
\end{remark}

\begin{lemma}\label{lemma:regularity}
  Assume Hypothesis \ref{phi_conditions_3} and let $x \in L^2(\L)$, $s \in \R$ and $\o \in \O$. Then $\Phi(S(\cdot,s,\o)x) \in L^{\frac{p+1}{p}}_{loc}([s,\infty); H_0^{1,\frac{p+1}{p}})$ and there exist constants $c > 0, C \in \R$, independent of $x,s$ and $\o$, such that 
\begin{align*}
 & \|Z(t_2,s,x,\o)\|_2^2 + c \int_{t_1}^{t_2} \|\nabla \Phi(S(r,s,\o)x)\|_{\frac{p+1}{p}}^{\frac{p+1}{p}} dr \\ \le & \|Z(t_1,s,x,\o)\|_2^2 + C \int_{t_1}^{t_2} (\|\nabla QW_r(\o)\|_{p+1}^{p+1} + 1) dr,\ \ \forall t_2 \ge t_1 \ge s. 
\end{align*}
\end{lemma}
\begin{proof}
  We use the Galerkin approximation and the notation used in the proof of unique existence of a solution to \eqref{eqn:o-wise} in \cite[Theorem 4.2.4]{pr}). Let $\{ e_i | i \in \N \}$ be the orthonormal basis of $H$ consisting of eigenfunctions of $\D$ on $L^2(\L)$ with Dirichlet boundary conditions. Then $e_i \in C^\infty(\L) \cap H_0^1(\L) \subseteq V$. Furthermore, let $H_n=$ span$\{e_1,...,e_n\}$ and define $P_n: V^* \to H_n \subseteq C^\infty(\L) \cap H_0^1(\L)$ by
    \[ P_n y := {\sum_{i=1}^n}\ _{V^*}\langle y,e_i \rangle_V\ e_i. \]
  Note that via the embedding $L^2(\L) \subseteq H \subseteq V^*$, ${P_n}_{|L^2(\L)}: L^2(\L) \to H_n$ is just the orthogonal projection in $L^2(\L)$ onto $H_n$.
   Let $t_1 \ge s$ such that $Z_{t_1} \in L^2(\L)$, let $Z^n_t$ denote the solution of
    \[ Z^n_t = P_n Z_{t_1} + \int_{t_1}^t P_n A_\o(r,Z^n_r) dr,\ \forall t \ge {t_1}\]
  and let $S^n_t := Z^n_t + QW_t$. By the chain rule, for all $t_2 \ge t_1$
  \begin{align}\label{eqn:galerkin_1}
    \|Z^n_{t_2}\|_2^2 
    &= \|P_n Z_{t_1}\|_2^2 + 2 \int_{t_1}^{t_2} \left< A_\o(r,Z^n_r),Z^n_r \right> dr \\
    &= \|P_n Z_{t_1}\|_2^2 + 2 \int_{t_1}^{t_2} \left< \D \Phi(S^n_r),S^n_r \right> dr - 2 \int_{t_1}^{t_2} \left< \D \Phi(S^n_r),QW_r \right> dr . \nonumber
  \end{align}
  By the same argument as for \eqref{2.9} and with $\zeta$ as defined in Lemma \ref{lemma:hyp1.4->hyp1.1} we get
  \begin{align*}
\left< \D \Phi(S^n_r),S^n_r \right> = - \left< \nabla \Phi(S^n_r),\nabla S^n_r \right> = - \left< \sqrt{\Phi'(S^n_r)} \nabla \zeta(S^n_r),\nabla S^n_r \right> = - \|\nabla \zeta(S^n_r)\|_2^2
  \end{align*}
  and using Young's inequality
  \begin{align*}
-\left< \D \Phi(S^n_r),QW_r \right> = \left< \nabla \Phi(S^n_r), \nabla QW_r \right> \le \ve \|\nabla \Phi(S^n_r)\|_{\frac{p+1}{p}}^{\frac{p+1}{p}} + C_\ve \|\nabla QW_r\|_{p+1}^{p+1},
  \end{align*}
  for all $\ve > 0$ and some $C_\ve \in \R$. By \eqref{eqn:galerkin_1} this yields
  \begin{align}\label{eqn:galerkin_2}
    \|Z^n_{t_2}\|_2^2 
    &\le \|P_n Z_{t_1}\|_2^2 - 2 \int_{t_1}^{t_2} \|\nabla \zeta(S^n_r)\|_2^2 dr \\
    &\hskip13pt+ 2 \ve \int_{t_1}^{t_2} \|\nabla \Phi(S^n_r)\|_{\frac{p+1}{p}}^{\frac{p+1}{p}} dr + 2 C_\ve \int_{t_1}^{t_2} \|\nabla QW_r\|_{p+1}^{p+1} dr \nonumber.
  \end{align}
  By the same argument as for \eqref{eqn:bound_phi_zeta} we realize
  \begin{align*}
    \|\nabla \Phi(S^n_r)\|_{\frac{p+1}{p}}^{\frac{p+1}{p}} \le C_1 \|\nabla \zeta(S_r^n)\|_2^2 + C_2,
  \end{align*}
  for some constants $C_1,C_2$. Using this in \eqref{eqn:galerkin_2}, with $\ve = \frac{1}{2C_1}$ yields for some $c > 0, C \in \R$
  \begin{align}\label{eqn:galerkin_3}
    \|Z^n_{t_2}\|_2^2 + c \int_{t_1}^{t_2} \|\nabla \Phi(S^n_r)\|_{\frac{p+1}{p}}^{\frac{p+1}{p}} dr \le \|Z_{t_1}\|_2^2 + C \int_{t_1}^{t_2} (\|\nabla QW_r\|_{p+1}^{p+1} + 1) dr.
  \end{align}
Both $C_1,C_2$ and $c,C$ are independent of $x,s$ and $\o$.\\
  Hence we obtain the existence of a $\bar{\phi} \in L^\frac{p+1}{p}([t_1,t_2];H_0^{1,\frac{p+1}{p}})$ such that (selecting a subsequence if necessary)
    \[ \Phi(S^n_r) \rightharpoonup \bar{\phi}, \]
  in $L^\frac{p+1}{p}([t_1,t_2];H_0^{1,\frac{p+1}{p}})$ and thus in $L^\frac{p+1}{p}([t_1,t_2];L^{\frac{p+1}{p}}(\L))$. By the proof of unique existence of a solution we also know that (again selecting a subsequence if necessary)
    \[ \D\Phi(S_r^n) \rightharpoonup \D\Phi(S_r),\]
  in $L^\frac{p+1}{p}([t_1,t_2];V^*)$ and by definition of $\D\Phi: V \to V^*$ this is equivalent to $\Phi(S_r^n) \rightharpoonup \Phi(S_r)$,
  in $L^\frac{p+1}{p}([t_1,t_2];L^{\frac{p+1}{p}}(\L))$. Hence $\bar{\phi} = \Phi(S_r)$. An analogous argument applied to $Z^n_{t_2}$ yields $Z^n_{t_2} \rightharpoonup Z_{t_2}$ in $L^2(\L)$. Letting $n \to \infty$ in \eqref{eqn:galerkin_3} we arrive at
  \begin{align}\label{eqn:galerkin_4}
    \|Z_{t_2}\|_2^2 + c \int_{t_1}^{t_2} \|\nabla \Phi(S_r)\|_{\frac{p+1}{p}}^{\frac{p+1}{p}} dr \le \|Z_{t_1}\|_2^2 + C \int_{t_1}^{t_2} (\|\nabla QW_r\|_{p+1}^{p+1} + 1) dr.
  \end{align}
  Since $Z_s = x-QW_s \in L^2(\L)$, for all $t_1 \ge s$ we obtain $Z_{t_1} \in L^2(\L)$ and thus \eqref{eqn:galerkin_4} holds for all $t_2 \ge t_1 \ge s$. 
\end{proof}

\begin{corollary}[Compact absorption]\label{cor:compact-absorption}
   There is an $\F$--measurable function $\kappa : \O \to \R_+$ such that for each $\varrho >0$ there exists $s(\varrho ) \leq -1 $ such that for all $x\in H$ with $\norm{x}_H\leq \varrho$ and all $\o \in \O_0$
   \[\norm{ S(0,s,\o)x}_{2}\leq \kappa(\o), \text{ for all } s \leq s (\varrho ).\]
\end{corollary}
\begin{remark}
 This is analogous to \cite[Lemma 5.5, p. 380]{cf1}.
\end{remark}
\begin{proof}
 \eqref{eqn:L2-bound} in Theorem 3.1 with $t_2 = 0 \ge t_1 \ge s$ implies 
 \begin{align*}
    \norm {Z_0}_{2}^2 
      &\leq \norm {Z_{t_1}}_{2}^2 - \a \int_{t_1}^0 \left(\norm {Z_r}_{2}^2 + p_2^{(\a)}(r,\o)\right)dr.
 \end{align*}
 Integrating over $t_1 \in [-1,0]$ yields
 \begin{align*}
    \norm {Z_0}_{2}^2 
      &\leq \int_{-1}^0 \left(\norm {Z_r}_{2}^2 + |p_2^{(\a)}(r,\o)|\right)dr \\
      &\leq \int_{-1}^0 (2\norm {S_r}_{2}^2 + 2\norm {QW_r}_{2}^2 + |p_2^{(\a)}(r,\o)|)dr.
 \end{align*}
 Hence using Corollary \ref{cor:aux_estimate} and recalling that $Z_0 = S(0,s,\o)x$ we obtain the assertion.
\end{proof}


\section{Existence of the global random attractor}

\begin{theorem}\label{thm:existence_random_attractor}
  The random dynamical system associated with \eqref{e1.0} and defined by \eqref{eqn:def_rds} admits a random attractor.
\end{theorem}
\begin{proof}
 We show that the assumptions of Proposition \ref{prop:suff_criterion} are satisfied. Since the embedding $L^2(\L) \hookrightarrow H$ is compact, for each $\o \in \O$ the set
   \[K (\o) := \overline{B_{L^2} (0, \kappa (\o ) )}^H\]
 is nonempty and compact in $H$. 
 
 For the reader's convenience, we prove that it is a random set (cf. Definition \ref{def:rds_basics} (i)) in the Polish space $H$. According to \cite[Proposition 2.4]{c}, it is enough to check that for each open set $O \subset H$, $C_{O}:= \{ \o \in \O | O \cap K(\o)  \neq \emptyset\}$ is measurable. But
 \begin{align*}
   O \cap K(\o) =&O  \cap \overline {B_{L^2} (0, \kappa (\o) )} ^H = O  \cap \bar { B}_{L^2} (0, \kappa (\o)) \\
   = &O  \cap L^2(\L) \cap \bar { B}_{L^2} (0, \kappa (\o))  
 .\end{align*}
For $C \subseteq L^2(\L)$ and $x \in L^2(\L)$ let $d_{L^2}(x,C) := \inf \limits_{y \in C} \|x-y\|_2$. If $O \cap {L^2}(\L) = \emptyset$, then $C_{O}= \emptyset$ is measurable and if $O \cap {L^2}(\L) \neq \emptyset$, then
  \[C_{O}= \{ \o \in \O| d_{L^2}\left(0, O \cap {L^2}(\L) \right) \le \kappa ( \o)\} \]
is measurable as $\kappa $ is.

Let $B$ be a bounded subset of $H$. Then $B \subset \bar { B}_H(0,\varrho)$, for some $\varrho>0$. By Corollary  \ref{cor:compact-absorption} there exists a $t_B := - s(\varrho ) \ge 1$ such that for all $x \in B$, $t \ge t_B$ and $\o \in \O_0$
  \[\vp (t, \theta _{-t}\o ) (x) = S(t, 0 , \theta_{-t} \o )x = S(0,-t,\o)x \le \k(\o).\]
Hence for all $t \ge t_B, \o \in \O_0$, $\vp (t, \theta _{-t}\o)(B) \subset K(\o)$, i.e. the random compact set $K$ absorbs all deterministic bounded sets.

Now we may apply Proposition \ref{prop:suff_criterion} to get the existence of a global compact attractor $A$, given by:
  \[A(\o) = \overline{\bigcup_{B\subset H, \, B \text{ bounded}} \O_B(\o)}^H,\]
where $\O_B(\o) := \bigcap \limits_{T \ge 0} \overline{\bigcup \limits_{t \ge T} \vp(t,\t_{-t}\o)B}$ denotes the $\O$-limit set of $B$.
\end{proof}

\begin{remark}
  By \cite[Proposition 4.5]{cf1} the existence of a random attractor as constructed in the proof of Theorem \ref{thm:existence_random_attractor} implies the existence of an invariant Markov measure $\mu_\cdot \in \mathcal P_\O(H)$ for $\vp$  (in the sense of \cite[Definition 4.1]{cf1}), supported by $A$. Hence using \cite{c-mm} there exists an invariant measure for the Markovian semigroup defined by $P_tf(x) = \E [f(S(t,0,x))]$ and it is given by
    \[\mu (B) = \int _\O \mu_ \o (B) P(d\o),\]
  where $B \subseteq H$ is a Borel set. If the invariant measure $\mu$ for $P_t$ is unique, then the invariant Markov measure $\mu_\cdot$ for $\vp$ is unique and given by
    \[ \mu_\o = \lim_{t \to \infty} \vp(t,\t_{-t}\o)\mu. \]
\end{remark}


\section{Attraction by a single point}

\parskip10pt

So far we obtained the existence of the random attractor $A$ for \eqref{e1.0}, but we did not deduce any information about its finer structure. Under a stronger monotonicity condition which was first introduced in \cite{dprrw} we will now prove that $A$ consists of a single random point. While we had to restrict to noise of regularity at least $H_0^{1,p+1}(\L)$ before, we can now allow more general noise. Let $Q$ be a Hilbert-Schmidt operator from $L^2(\L) \to H$ and $W_t$ be a cylindrical Brownian Motion on $L^2(\L)$. Then $QW_t$ is an $R:= QQ^*$-Wiener process on $H$. Let $e_k$ be an orthonormal basis of eigenvectors of $R$ with corresponding eigenvalues $\mu_k$. Assume further $\sum_{k = 1}^\infty \sqrt{\mu_k}\|e_k\|_V < \infty$. Then $QW_t$ defines an almost surely continuous process in $V$. Now the associated RDS to \eqref{e1.0} can be defined as before.

Define $\Phi: \R \rightarrow \R$ to be a continuous function such that there exist some constants $c \ge 0$, $p \in (1,\infty)$, $\eta > 0$ such that
\begin{align}\label{eqn:monotone}
  &|\Phi(s)| \le c(1+|s|^p) \\
  &(s-t)(\Phi(s) - \Phi(t)) \ge \eta |s-t|^{p+1}, \hskip10pt s,t \in \R.\nonumber
\end{align}
It has been shown in \cite{dprrw} that \eqref{eqn:monotone} holds if $\Phi \in C^1(\R)$, $\Phi(0)=0$ and if there exist constants $\kappa, \eta > 0$ such that
\begin{equation}\label{eqn:strong_monotone}
  \frac{(p+1)^2}{4} \eta |s|^{p-1} \le \Phi'(s) \le \kappa (1+|s|^{p-1}), \ \ \ \ s \in \R.
\end{equation} 
This, for example is true for $\Phi(s)=s|s|^{p-1}$. By Remark \ref{rmk:phi_diff} it is easy to see that \eqref{eqn:strong_monotone} implies the weaker monotonicity assumption (A1)'. Also note that \eqref{eqn:monotone} implies the coercivity property (A2). Thus (A1)-(A3) are satisfied and we can define $Z_t, S_t$ and the RDS $\vp$ as before (cf. \eqref{eqn:def_rds}).

\begin{remark}\label{rmk:strong_monotone}
  \begin{enumerate}
   \item Obviously, for $\Phi$ as in Example \ref{exam:Phi} the conditions in \eqref{eqn:monotone} do not hold.
   \item Let $\Phi(r) := \int_0^r e^{-\frac{1}{|s|}}ds$. Then $\Phi \in C^1(\R)$, $0 < \Phi'(r) \le 1$ and Hypothesis \ref{phi_conditions_3} (hence also Hypothesis \ref{phi_conditions_2}) holds with $p = 1$. But obviously \eqref{eqn:monotone} above does not hold.
  \end{enumerate}

\end{remark}

\parskip10pt

\begin{theorem}\label{thm:main}
  Assume \eqref{eqn:monotone}. Then
    \begin{align*}
      \|S(t,s_1,\o)x-S(t,s_2,\o)y\|_H^2
      &\le \left\{\|S(s_2,s_1,\o)x-y\|_H^{1-p}+\eta \lambda_1^{\frac{p+1}{2}} (p-1) (t-s_2) \right\}^{-\frac{2}{p-1}} \\
      &\le \left\{\eta \lambda_1^{\frac{p+1}{2}} (p-1) (t-s_2)\right\}^{-\frac{2}{p-1}},
    \end{align*}
    for $s_1 \le s_2 < t$, $\o \in \O$ and $x,y \in H$. In particular for each $t \in \R$, $\lim \limits_{s \rightarrow -\infty}S(t,s,\o)x = \eta_t(\o)$ exists independently of $x$ and uniformly in $x, \omega$.
\end{theorem}
\begin{proof}
  Let $s_1 \le s_2 < t$. Then for all $s_2 \le s \le t$ 
   \[ S(t,s_1,\o)x-S(t,s_2,\o)y = S(s,s_1,\o)x - S(s,s_2,\o)y + \int_{s}^t \D\Phi  (S(r,s_1,\o)x)- \D\Phi  (S(r,s_2,\o)y) dr \nonumber.\]
  By It\^{o}'s-Formula and since $\|u\|_{p+1}^{p+1} \ge \lambda_1^{\frac{p+1}{2}} \|u\|_H^{p+1}$, for all $s_2 \le s \le t$: 
  \begin{align}\label{eqn:diff2}
    	& \|S(t,s_1,\o)x- S(t,s_2,\o)y\|_H^2 \nonumber \\
	&= \|S(s,s_1,\o)x - S(s,s_2,\o)y\|_H^2 \nonumber \\
	  &\hskip13pt + 2 \int_{s}^t \ _{_{V*}}\left< \D\Phi (S(r,s_1,\o)x)- \D\Phi (S(r,s_2,\o)y),S(r,s_1,\o)x-S(r,s_2,\o)y \right>_{_V} dr \nonumber\\
	&= \|S(s,s_1,\o)x - S(s,s_2,\o)y\|_H^2  \\
          &\hskip13pt - 2 \int_{s}^t  \ _{_{V*}}\left< \Phi(S(r,s_1,\o)x)-\Phi(S(r,s_2,\o)y),S(r,s_1,\o)x-S(r,s_2,\o)y \right>_{_V} dr \nonumber\\
	&\le \|S(s,s_1,\o)x - S(s,s_2,\o)y\|_H^2 - 2\eta \int_{s}^t  \|S(r,s_1,\o)x-S(r,s_2,\o)y\|_{p+1}^{p+1} dr \nonumber\\
	&\le \|S(s,s_1,\o)x - S(s,s_2,\o)y\|_H^2 - \tilde{\eta} \int_{s}^t \|S(r,s_1,\o)x-S(r,s_2,\o)y\|_H^{p+1} dr  \nonumber,
  \end{align}
  where for notational convenience we have set $\tilde{\eta} := 2\eta \lambda_1^{\frac{p+1}{2}}$. Thus, formally $\|S(t,s_1,\o)x-S(t,s_2,\o)y\|_H^2$ is a subsolution of the ordinary differential equation
  \begin{align}\label{def:ode}
    h'(t) &= -\tilde{\eta} h(t)^{\frac{p+1}{2}}, \hskip5pt \forall t \ge s_2 \\
    h(s_2) &= \|S(s_2,s_1,\o)x- y\|_H^2 \nonumber.
  \end{align}
  Let   
   \[ h_{\epsilon}(t) = \left\{(\|S(s_2,s_1,\o)x- y\|_H + \epsilon)^{1-p} + \frac{\tilde{\eta}}{2} (p-1) (t-s_2) \right\}^{-\frac{2}{p-1}},\ t \ge s_2.\]
  $h_{\epsilon}$ is a solution of \eqref{def:ode} with $h_\epsilon(s_2)=(\|S(s_2,s_1,\o)x- y\|_H + \epsilon)^2$, which suggests $\|S(t,s_1,\o)x-S(t,s_2,\o)y\|_H^2 \le h_{\epsilon}(t)$. This will be proved next.

  Let $\varPhi_\epsilon(t) := h_{\epsilon}(t) - \|S(t,s_1,\o)x-S(t,s_2,\o)y\|_H^2$ and $  \tau_\epsilon = \inf\left\{t \ge s_2 |\ 0 \ge \varPhi_\epsilon(t)\right\}$. Using $0 < \varPhi_\epsilon(s_2)$ and continuity of $\varPhi_\epsilon$ we realize $\tau_\epsilon > s_2$. Further note that by definition we have $h_{\epsilon}(t) \ge \|S(t,s_1,\o)x-S(t,s_2,\o)x\|_H^2$ on $[s_2,\tau_\epsilon]$ and that
  \[  h_{\epsilon}(t) \le (\|S(s_2,s_1,\o)x- y\|_H + \epsilon)^2 =: c_\epsilon.\]
  
  Assume $\tau_\epsilon < \infty$. Then $\varPhi_\epsilon(\tau_\epsilon) \le 0$ and for all $s_2 \le s \le t \le \tau_\epsilon$, by the mean value theorem and \eqref{eqn:diff2}:
  \begin{align*}
  \varPhi_\epsilon(t)	
    &= h_{\epsilon}(t) - \|S(t,s_1,\o)x-S(t,s_2,\o)y\|_H^2 \nonumber\\
    &\ge \varPhi_\epsilon(s) - \tilde{\eta} \int_{s}^t (h_\epsilon(r)^{\frac{p+1}{2}} - \left( \|S(r,s_1,\o)x-S(r,s_2,\o)\|_H^2\right)^{\frac{p+1}{2}}) dr \\
    &\ge \varPhi_\epsilon(s)  - \tilde{\eta} \left( \frac{p+1}{2} \right) c_\epsilon^{\frac{p-1}{2}} \int_{s}^t \varPhi_\epsilon(r) dr\nonumber.
  \end{align*}
  Using the Gronwall Lemma we obtain
  \begin{equation*}
  \varPhi_\epsilon(\tau_\epsilon) \ge \varPhi_\epsilon(s_2) e^{- \tilde{\eta} \left( \frac{p+1}{2} \right) c_\epsilon^{\frac{p-1}{2}} (\tau_\epsilon-s_2)} > 0.
  \end{equation*}
  This contradiction proves $\tau_\epsilon = \infty$ and since this is true for all $\epsilon >0$ we conclude:
  \begin{align*}
    \|S(t,s_1,\o)x-S(t,s_2,\o)y\|_H^2 
      &\le \left\{(\|S(s_2,s_1,\o)x- y\|_H)^{1-p} + \frac{\tilde{\eta}}{2} (p-1) (t-s_2)\right\}^{-\frac{2}{p-1}} \nonumber \\
    &\le \|S(s_2,s_1,\o)x- y\|_H^2 \wedge \left\{\frac{\tilde{\eta}}{2} (p-1) (t-s_2)\right\}^{-\frac{2}{p-1}} \nonumber \\
    &\le \left\{\frac{\tilde{\eta}}{2} (p-1) (t-s_2)\right\}^{-\frac{2}{p-1}},
  \end{align*}
  for each $t > s_2$.

\end{proof}

\begin{theorem}
  Assume \eqref{eqn:monotone}. The random dynamical system given by $\vp(t,\o)x=S(t,0,\o)x$ has a compact global attractor $A(\o)$ consisting of one point
  \begin{equation*}
    A(\o) = \{\eta_0(\o)\}.
  \end{equation*}
\end{theorem}
\begin{proof}
  Since $\eta_0(\o)$ is measurable, $A(\o)$ is a random compact set. We need to check invariance and attraction for $A(\o)$. Let $t>0$. Then for any $x \in H$, by continuity of $x \mapsto S(t,0,\o)x$ and \eqref{1.6'}, \eqref{1.6''}
  \begin{align*}
    \vp(t,\o)A(\o) 
      &= \left\{ S(t,0,\o) \lim_{s \rightarrow -\infty}S(0,s,\o)x \right\} = \left\{ \lim_{s \rightarrow -\infty} S(t,s,\o)x \right\} \\
      &= \left\{ \lim_{s \rightarrow -\infty} S(0,s-t,\t_t\o)x \right\} = \{ \eta_0(\t_t\o) \} = A(\t_t\o)
  .\end{align*}
  Since the convergence in Theorem \ref{thm:main} is uniform with respect to $x \in H$, for any bounded set $B \subseteq H$ we have (again using \eqref{1.6''})
  \begin{align*}
      d(\vp(t,\t_{-t}\o)B,A(\o)) 
          &= \sup_{x \in B} \|S(t,0,\t_{-t}\o)x - \eta_0(\o)\|_H \\
          &= \sup_{x \in B} \|S(0,-t,\o)x - \eta_0(\o)\|_H \rightarrow 0
  ,\end{align*}
  for $t \rightarrow \infty$. Hence $A(\o)$ attracts all deterministic bounded sets.

\end{proof}

It is easy to see that the convergence $\lim \limits_{s \rightarrow -\infty}S(t,s,\o)x = \eta_t(\o)$ implies the existence and uniqueness of an invariant measure for the associated Markovian semigroup, defined by $P_tf(x) := \E[ f(S(t,0,\cdot)x) ]$ (cf. \cite{dprrw}). This invariant measure is given by $\mu = \P \circ \eta_0^{-1}$. In fact we can deduce much more. Since evidently $\eta_0$ is measurable with respect to $\mcF^{-}$ (as defined in \cite{c-mm}), by \cite{c-mm} $\mu_\o := \lim \limits_{t \to \infty} \vp(t,\t_{-t}\o)\mu$ exists $\P$-a.s. and defines an invariant measure for the random dynamical system $\vp$ (for more details on invariant random measures cf. \cite{cf1}). Moreover by \cite[Theorem 2.12]{c} every invariant measure for $\vp$ is supported by $A=\{ \eta_0 \}$, i.e. $\mu_\o(\{ \eta_0(\o) \}) = 1$ for $\P$-a.a. $\omega$. Hence we have proved the following

\begin{corollary}
  There exists a unique invariant random measure $\mu_\cdot \in \mcP_\O(H)$ for the random dynamical system $\vp$ and it is given by
  \[ \mu_\o = \delta_{\eta_0(\o)}, \hskip10pt \P \text{-a.s. .}\]
\end{corollary}


\section{Concluding remarks on computational approaches}

The porous medium equation considered here is a model case for a 
general type of equations that include more details of the permeable
medium and that has important applications to the simulation of
oil reservoirs. We refer to \cite{AE08} for such an application and
for an up-to-date finite element method that can be used for solving
the deterministic version of (1.1). One of the major difficulties
here is to account for the spatial variations (represented
by the functions $\varphi_j$ in the operator $Q$) by introducing different
scales in the finite element subspace. For the quasilinear steady state
equation suitable finite element approximations have been set up, cf.
\cite{Ma08}, \cite{MK05} and the references therein.

It seems, however, that computational methods for random attractors
in infinite dimensional systems (except for the case of a singleton)
are well beyond today's computational capabilities.

There are a few approaches to approximate random attractors in 
stochastic ordinary differential equations \cite{J07}, \cite{KO99}, \cite{T05}.
These are based on the subdivision and box covering techniques developed over 
the last years by  Dellnitz and coworkers (see \cite{DJ02} for a survey).
However, these methods are essentially still limited to lower dimensions.
In order to proceed to high-dimensional or even infinite-dimensional cases
(see e.g. \cite{Y06}) one
will need reduction principles as they are well established in the theory
of inertial manifolds for deterministic PDEs.
The corresponding properties of squeezing and flattening
(cf. \cite{FT79},\cite{MWZ02}) have been 
generalized to random dynamical systems in \cite{KL07}. It is also shown
in \cite{KL07} that squeezing is a stronger condition than flattening,
but that the latter one is sufficient to establish the existence
of a compact random attractor. The determining modes occuring in these
properties should form the basis of a reduced space to which
numerical methods apply.

\end{document}